\documentclass[11pt]{amsart}
\usepackage{verbatim, amscd, epsfig,amsmath,amsthm,amstext,amssymb,hyperref,tikz}
\usepackage[latin1]{inputenc}
\usepackage{amssymb, booktabs, longtable}
\usepackage[margin=3.8cm]{geometry}
\usepackage[ruled,vlined,norelsize]{algorithm2e}

\theoremstyle{plain}
\newtheorem*{theorem*}{Theorem}
\newtheorem{theorem}{Theorem}[section]

\newtheorem{prop}[theorem]{Proposition}
\newtheorem{thm}[theorem]{Theorem}
\newtheorem{lem}[theorem]{Lemma}

\theoremstyle{definition}
\newtheorem{definition}[theorem]{Definition}

\newcommand{\X}{\mathbb{ X}}
\newcommand{\R}{\mathbb{ R}}

\newcommand{\D}{\mathcal{D}}
\newcommand{\E}{\mathbb{E}}

\newcommand{\bb}{\mathrm{b}}
\newcommand{\dd}{\mathrm{d}}

\newcommand{\diag}{\Delta}
\begin{document}
\title
{Fr\'echet Means for Distributions of Persistence diagrams}
\author{Katharine Turner$^{1}$, Yuriy Mileyko$^{2}$, Sayan
  Mukherjee$^{3}$, John Harer$^{2}$}
\address{$^1$  Department of Mathematics, 
University of Chicago}
\address{$^2$  Departments of Statistical Science, Computer Science,
and Mathematics, Institute for Genome Sciences \& Policy, 
Duke University}
\address{$^3$Departments of Mathematics and Computer Science, Center
  for Systems Biology, Duke University}
\date{\today}

\begin{abstract}
Given a distribution $\rho$ on persistence diagrams and observations
$X_1,...X_n \stackrel{iid}{\sim} \rho$ we introduce an algorithm in this paper
that estimates a Fr\'echet mean from the set of diagrams $X_1,...X_n$. If the underlying measure $\rho$ is
a combination of Dirac masses $\rho = \frac{1}{m} \sum_{i=1}^m \delta_{Z_i}$
then we prove
the algorithm converges to a local minimum and a law of large numbers result for a Fr\'echet mean computed by
the algorithm given observations drawn iid from $\rho$. We illustrate the convergence
of an empirical mean computed by the algorithm to a population mean by simulations
from Gaussian random fields.
 \end{abstract}
\maketitle
\section{Introduction}
There has been a recent effort  in topological data analysis (TDA) to incorporate ideas from stochastic modeling. 
Much of this work involved the study of random abstract simplicial complexes generated from stochastic processes
\cite{Penr:2003,Penr:Yuki:2001,Kahle:2011,Kahle:2009,Luna:2009,KahleMeckes} and non-asymptotic bounds on 
the convergence or consistency of topological summaries  as the number of points increase \cite{NiySmaWei2008,NiySmaWei2008b,ChaCohLie2009,BubCarKimLuo2010,BenMukWang2012}. The central idea in 
these papers has been to study statistical properties of topological 
summaries of point cloud data. 

In \cite{MilMukHar:2012} it was shown that a commonly used topological summary, the persistence diagram \cite{EdeHar2010}, admits a well 
defined notion of probability distributions and notions such as expectations, variances, percentiles and conditional probabilities. 
The key contribution of this paper is characterizing Fr\'echet means and variances of finitely many persistence diagrams and providing an algorithm for estimating them. Existence  
of these means and variances was previously shown.  However, a procedure to compute means and variances was not provided.  

In this paper we state an algorithm which when given an observed set of persistence diagrams  $X_1,...,X_n$ computes a new diagram which is a local
minimum of the Fr\'echet function of the empirical measure corresponding to the empirical distribution $\rho_n := n^{-1} \sum_{i=1}^n \delta_{X_i}$.
In the case where the diagrams are sampled independently and identically from a probability measure that is a finite 
combination of Dirac masses we provide a (weak) law of large numbers for the local minima  computed by the algorithm we propose.

\section{Persistence diagrams and Alexandrov spaces with curvature bounded from below}
In this section we state properties of the space of persistence
diagrams that we will use in the subsequent sections.
We first define persistence diagrams and the $L^2$-Wasserstein
metric on the set of persistence diagrams. Note that this is not the
same metric as was used in \cite{MilMukHar:2012}. We discuss the relation 
between the two metrics and why we work with the $L^2$-Wasserstein 
metric later in this section. We then show that the space of
persistence diagrams is a geodesic space and specifically an 
Alexandrov space with curvature bounded from below. We show that the 
Fr\'{e}chet function in this space is semiconcave which allows us to 
define supporting vectors which will serve as an analog of the gradient.
The supporting vectors will be used in the algorithm developed in the 
following section to find local minima -- the algorithm is a gradient 
descent based method.

\subsection{Persistent homology and persistence diagrams}

Consider a topological space $\X$ and a bounded continuous function
$f: \X \rightarrow \R$. For a threshold $a$ we define sublevel sets 
$\X_a = f^{-1}(-\infty,a]$. For $a \leq b$ inclusions $\X_a \subset  \X_b$
induce homomorphisms of the homology groups of sublevel sets:
$$\mathbf{f}_\ell^{a,b}: \mathbf{H}_\ell(\X_a) \rightarrow \mathbf{H}_\ell(\X_b),$$
for each dimension $\ell$. We assume the function $f$ is tame which means
that $\mathbf{f}_\ell^{c-\delta,c}$ is not an isomorphism for any $\delta>0$ 
at only a finite number of $c$'s for all dimensions $\ell$ and $\mathbf{H}_\ell(\X_a)$
is finitely generated for all $a \in \R$. We also assume that the homology
groups are defined over field coefficients, e.g. $\mathbb{Z}_2$.

By the tameness assumption the image $\mathbf{F}_\ell^{a-,b} := \mbox{Im} \mathbf{f}_\ell^{a-\delta,b} \subset \mathbf{H}_\ell(\X_b)$ is independent of $\delta>0$ if $\delta$ is small enough.  The quotient group
$$\mathbf{B}_\ell^{a} = \mathbf{H}_\ell(\X_a) / \mathbf{F}_\ell^{a-,a}$$
is the cokernel of $\mathbf{f}_\ell^{a-\delta,a}$ and captures homology
classes which did not exist in sublevel sets preceding $\X_a$. This group
is called the $\ell$-th birth group at $\X_a$ and we say that a homology class 
$\alpha \in \mathbf{H}_\ell(\X_a)$ is born at $\X_a$ if its projection
onto $\mathbf{B}_\ell^{a}$ is nontrivial.

Consider the map 
$$\mathbf{g}_\ell^{a,b}: \mathbf{B}_\ell^a \rightarrow \mathbf{H}_\ell(\X_b) / \mathbf{F}^{a-,b}$$ 
and denote its kernel as $\mathbf{D}_\ell^{a,b}$. The kernel captures homology classes
that were born at $\X_a$ but at $\X_b$ are homologous to homology classes born before
$\X_a$. We say that a homology class $\alpha \in \mathbf{H}_\ell(\X_a)$ that was born at $\X_a$ dies entering $\X_b$
if its projection onto $\mathbf{D}_\ell^{a,b}$ is $0$ but its projection to $\mathbf{D}_{\ell}^{a,b-\delta}$ is nontrivial for all sufficiently small $\delta>0$. We also
call $b$ a degree-$r$ death value of $\mathbf{B}_\ell^a$ if
$\mathrm{rank}\mathbf{D}_\ell^{a,b}-\mathrm{rank}\mathbf{D}_\ell^{a,b-\delta}=r>0$
for all sufficiently small $\delta>0$.

If a homology class $\alpha$ is born at $\X_a$ and dies entering $\X_b$
we set $\bb(\alpha) = a$ and $\dd(\alpha) = b$ and represent the births and deaths
of $\ell$-dimensional homology classes by a multiset of points
in $\R^2$ with the horizontal axis corresponding to the birth of a
class, the vertical axis corresponding to the death of a class, and the
multiplicity of a point being the degree of the death value. The idea of a persistence diagram is to consider a basis of persistent homology classes $\{\alpha\}$ and to represent each persistent homology class $\alpha$ by a point $(b(\alpha), d(\alpha))$.

The persistence
of $\alpha$ is the difference $\mbox{pers}(\alpha) = \dd(\alpha) - \bb(\alpha)$. In the general setting we could have points with infinite persistence which corresponds to points of the form $(-\infty, y)$ or $(x,\infty)$. These points are infinitely far from all points on finite persistence and hence would have to be treated separately. The space of persistence diagrams would be forced to be disconnected with each component corresponding to the number of points at infinity. For the sake of clarity we will restrict ourselves to the case where all classes have finite persistence. This can be achieved by considering extended persistence but for simplicity we can simply kill everything by setting $\mathbf{g}_\ell^{a,b} = 0$ if $b \geq \sup_{x \in \X} f(x)$.

After establishing some notation we can define persistence diagrams and the distance between two diagrams. Let $\Delta=\{(x,y)\in \R^2  \mid x=y\}$ be the diagonal in $\R^2$.  Let $\| x-y \|$ be the usual Euclidean distance if $x$ and $y$ are off diagonal points. With a slight abuse of notation let $\|x-\Delta\|$ denote the perpendicular distance between $x$ and the diagonal and $\|\Delta-\Delta\|=0$.

\begin{definition}
A persistence diagram is a countable multiset of points in
$\R^2$ along with the infinitely many copies of the diagonal $\Delta=\{(x,y)\in \R^2  \mid x=y\}$. 
We also require for the countably many points $x_j\in \R^2$ not lying on the
diagonal that $\sum_j \|x_j-\Delta\| <\infty$.
\end{definition}

Each point $p=(a,b)$ in a persistence diagram corresponds to some homology class $\alpha$ with $\bb(\alpha)=a$ and $\dd(\alpha)=b$. As a slight abuse of notation we say that $p$ is born at $\bb(p):=\bb(\alpha)$ and dies at $\dd(p):=\dd(\alpha)$.

We denote the set of all persistence diagrams by $\D$. One metric on $\D$ is
the $L^2$-Wasserstein metric 
\begin{equation}
\label{eq:defdistance}
d_{L^2}(X,Y)^2 = \inf_{\phi:X \to Y} \sum_{x\in X} \|x-\phi(x)\|^2
\end{equation}
Here we consider all the possible bijections $\phi$ between the off diagonal points and copies of the diagonal in $X$ and the  off diagonal points and copies of the diagonal in $Y$. Bijections always exist as any point can be paired to the diagonal. We will call a bijection \emph{optimal} if it achieves this infimum.

In much of the computational topology literature the following 
$p$-th Wasserstein distance between two persistence diagrams, 
$X$ and $Y$, is used
$$
d_{W_p}(X, Y)=\left( \inf_{\phi}\sum_{x\in X}{\|x-\phi(x)\|^p_{\infty}}
\right)^{\frac{1}{p}}.$$
In \cite{MilMukHar:2012} the above metric was used to define
the following space of persistence diagrams
$$
\D_p=\{x \mid d_{W_p}(x,\emptyset)<\infty\},
$$
with $p \geq 1$ and $\emptyset$ is the diagram with just the diagonal. 
It was shown in \cite{MilMukHar:2012}[Thm 6 and 10] that $\D_p$ is a complete separable metric space and probability measures on this
space can be defined. Given a probability measure $\rho$ on $\D_p$ the 
existence of a Fr\'echet mean was proven under restrictions on the space
of persistence diagrams $\D_p$ \cite{MilMukHar:2012}[Thm 21 and Lemma
27]. The basic requirement is that $\rho$ has a finite second moment
and the support of $\rho$ has compact support or is concentrated on a
set with compact support. 

In this paper we focus on the $L^2$-Wasserstein metric since 
it leads to a geodesic space with some known structure. Thus we consider the
space of persistence diagrams
$$\D_{L^2}=\{x \mid d_{L^2}(x,\emptyset)<\infty\}.$$ 
The results stated in the previous paragraph will also hold for
$\D_{L^2}$ with metric $d_{L^2}$, including existence
of Fr\'echet means. This follows from the fact that
for any $x,y \in \R^2$
\begin{equation}
\label{normeq}
\|x-y \|_\infty \leq \|x -y \|_2 \leq \sqrt{2} \|x - y \|_\infty,
\end{equation}
so $d_{W_2}(X,Y)\leq d_{L^2}(X,Y) \leq \sqrt{2}d_{W_2}(X,Y)$. This inequality coupled with   
the results in \cite{CohEdeHar2009c} implies
the following stability result for the $L^2$
Wasserstein distance.
\begin{thm}
Let $\X$ be a triangulable, compact
metric space such that $d_{W_k}(\mbox{Diag}(h), \emptyset)^k\leq C_{\X}$ for
any tame Lipschitz function $h:\X\to\R$ with Lipschitz constant $1$, where
$\mbox{diag}(h)$ denotes the persistence diagram of $h$, $k\in[1,2)$,
and $C_{\X}$ is a constant depending only on the space $\X$.
Then for two tame Lipschitz functions $f, g : \mathbb{X} \to \mathbb {R}$
we have
$$
d_{L^2}(\mbox{Diag}(f), \mbox{Diag}(g)) \leq 2^{\frac{k+2}{2}}\left [ C \|f -
g\|_{\infty}^{2-k}\right]^{\frac{1}{2}},
$$
where $C = C_{\X} \max\{\mbox{Lip}(f)^k,\mbox{Lip}(g)^k\}$.
\end{thm}

For ease of notation in the rest of the paper we denote
$d_{L^2}(X,Y)^2$ as $d(X,Y)^2$.

\begin{prop}\label{prop:infismin}
For any diagrams $X, Y\in \D_{L^2}$ the infimum in \eqref{eq:defdistance} 
is always achieved.
\end{prop}

We prove this proposition in the appendix.

We now show that the space of persistence diagrams with the above metric is a geodesic space. A rectifiable curve $\gamma : [ 0, l] \to X$ is called a geodesic if it is locally minimizing and parametrized proportionally to the arc length. If  $\gamma$ is also globally minimizing, then it is said to be minimal. $\D_{L^2}$ is a geodesic space if every pair of points is connected by a minimal geodesic. Now consider diagrams $X=\{x\}$ and $Y=\{y\}$ and some optimal pairing $\phi$ between the points in $X$ and $Y$. Let $\gamma:[0,1]\to \D_{L^2}$ be the path from $X$ to $Y$ where $\gamma(t)$ is the diagram with points which have travelled in a straight line from the point (which can be a copy of the diagonal) $x$ to the point (which can be a copy of the diagonal) for a distance of $t\|x-\phi(x)\|$. In other words, the diagram with points $\{(1-t)x+t\phi(x)\,|\, x\in X\}$.\footnote{If both $x$ and $\phi(x)$ are the diagonal then this is the diagonal. If exactly one of $x$ or $\phi(x)$ is the diagonal then we replace it in this sum by the closest point in the diagonal to $\phi(x)$ or $x$ respectively.} $\gamma$ is a geodesic from $X$ to $Y$. The proof of this is the observation that $\phi_t^X:X\to\gamma(t)$ where 
\begin{align}\label{eq:geobij}
\phi_t^X(x)=(1-t)x+t\phi(x)
\end{align}
is optimal. 

\subsection{Gradients and supporting vectors on $\D_{L^2}$}

We will propose a gradient descent based algorithm to compute Fr\'echet means. To analyze and understand the algorithm we
will need to understand the structure of $\D_{L^2}$. We will show 
that $\D_{L^2}$ is an Alexandrov space with curvature bounded from
below (see \cite{Buragoetal} for more information on these spaces).
This result is not so surprising since there are known relations 
between $L^2$-Wasserstein spaces and Alexandrov spaces with 
curvature bounded from below \cite{Ohta,LV}.
The motivating idea behind these spaces was 
to generalize the results of Riemannian geometry to metric spaces without
Riemannian structure.

The property and behavior of Fr\'echet means is closely
related to the curvature of the space. For metric spaces with curvature bounded from above, 
called $CAT$-spaces,\footnote{Terminology given by Gromov \cite{Grom87}  that
 stands for Cartan, Alexandrov, and Toponogov.}  properties of Fr\'echet means
 have  been investigated and there exist algorithms to compute Fr\'echet means \cite{Sturm}. $\D_{L^2}$ is
 not a $CAT$-space, see Proposition \ref{notCAT}.  $\D_{L^2}$ is however
an Alexandrov space with curvature bounded from below. Less is known
about properties of Fr\'echet means in these spaces as well as algorithms to
compute Fr\'echet means.  We use the structure of Alexandrov spaces with 
curvature bounded from below to compute estimates of Fr\'echet means
and provide some analysis of these estimates. Note that Fr\'echet means are the same as barycenters which is what is referred to in much of the literature.

We first confirm that $\D_{L^2}$ is not a $CAT$-space.

\begin{prop}\label{notCAT}
$\D_{L^2}$ is not in $\mbox{CAT}(k)$ for any $k>0$.
\end{prop}
\begin{proof}
If $\D_{L^2}  \in \mbox{CAT}(k)$ then for all $X,Y\in \D_{L^2}$ with
$d(X,Y)^2<\pi^2/k$ there is a unique geodesic between them 
\cite{BridsonHaefliger}[Proposition 2.11]. However, we can find 
$X,Y$ arbitrarily close with two distinct geodesics. One example is taking $X$ to be a diagram with two diagonally opposite corners of a square and $Y$ a diagram with the other two corners. The horizontal and vertical paths are equally optimal and we may choose the square to be as small as we wish.
\end{proof}

The following inequality characterizes Alexandrov spaces with 
curvature bounded from below by zero \cite{Ohta}. Given a geodesic space $\X$ with metric $d'$
for any geodesic $\gamma:[0,1] \to \X$ from $X$
to $Y$ and any $Z \in \X$ 
\begin{align}\label{eq:alex}
d'(Z,\gamma(t))^2 \geq  t d'(Z,Y)^2 + (1-t) d'(Z,X)^2  -t(1-t) d'(X,Y)^2.
\end{align}

We now show that $\D_{L^2}$ is a non-negatively curved Alexandrov space.

\begin{thm}
The space of persistence diagrams $\D_{L^2}$ with metric  $d$ given in \eqref{eq:defdistance} is a non-negatively curved Alexandrov space.
\end{thm}
\begin{proof}
First observe that $\D_{L^2}$ is a geodesic space. Let $\gamma:[0,1] \to \D_{L^2}$ be a geodesic from $X$ to $Y$ and let $Z \in \D_{L^2}$ be any diagram. We want to show that the inequality \eqref{eq:alex} holds.

Let $\phi$ be an optimal bijection between $X$ and $Y$ which induces the geodesic $\gamma$. That is $\gamma(t)=\{(1-t)x+t\phi(x)\,|\,x\in X\}$ and defined $\phi_t(x) =tx+(1-t)\phi(x)$ as done in  \eqref{eq:geobij}. Let $\phi_Z^t: Z \to \gamma(t)$ be optimal. Construct bijections $\phi_Z^X:Z\to X$ and $\phi_Z^Y: Z\to Y$ by 
$\phi_Z^X= (\phi_t)^{-1}\circ \phi_Z^t$ and $\phi_Z^Y=\phi \circ \phi_Z^X$. There is no reason to suppose that either bijections $\phi_Z^X$ or $\phi_Z^Y$ are optimal. Note that if $\phi_Z^t(z)=\diag$ then $\phi_Z^X(z)=\diag$ and $\phi_Z^Y(z)=\diag$.

From the formula for the distance in $\D_{L^2}$ we observe
\begin{equation}
\label{eq:ineqforgeo}
\begin{aligned}
d(Z,\gamma(t))^2&=\sum_{z\in Z} \|z-\phi_Z^t(z)\|^2=\sum_{z\in Z} \|z - [(1-t)\phi_Z^X(z)+t\phi_Z^Y(z)]\|^2, \\
d(Z,Y)^2&\leq \sum_{z\in Z}\|z-\phi_Z^Y(z)\|^2,\\
d(Z,X)^2&\leq \sum_{z\in Z}\|z-\phi_Z^X(z)\|^2,\\
d(X,Y)^2&=\sum_{z\in Z}\|\phi_Z^X(z)-\phi(\phi_Z^X(z))\|^2=\sum_{z\in Z}\|\phi_Z^X(z)-\phi_Z^Y(z)\|^2.
\end{aligned}
\end{equation}

Euclidean space has everywhere curvature zero so for each $z$ in the diagram $Z$, and all $t\in [0,1]$, we have
$$\|z - [(1-t)\phi_Z^X(z)+t\phi_Z^Y(z)]\|^2 =t\|z-\phi_Z^Y(z)\|^2 +(1-t)\|z-\phi_Z^X(z)\|^2-t(1-t)\|\phi_Z^X(z)-\phi_Z^Y(z)\|.$$

Combining these equalities with inequalities \eqref{eq:ineqforgeo} gives us the desired result.
%
%
\end{proof}

\subsection{Properties of the Fr\'echet function}

Given a probability distribution $\rho$ on $\D_{L^2}$ we can define the corresponding \emph{Fr\'echet function} to be
\begin{displaymath}
F:\D_{L^2} \to \R, \quad Y \mapsto \int_{\D_{L^2}} d(X,Y)^2 d\rho(X).
\end{displaymath}
The \emph{Fr\'echet mean set} of $\rho$ is the set of all the minimizers of the map $F$ on $\D_{L^2}$. If there is a unique minimizer then this is called the \emph{Fr\'echet mean}
 of $\rho$. The \emph{variance} is then defined to be the infimum of
 the above functional.

We will show that the Fr\'echet function has the nice property of being semiconcave. For an Alexandrov space $\Omega$, a locally Lipschitz function $f : \Omega \to \R$ is called \emph{$\lambda$-concave} if for any unit speed geodesic $\gamma$ in $\Omega$, the function
$$f \circ \gamma(t) - \lambda t^2/2$$
is concave. A function $f : \Omega\to \R$ is called \emph{semiconcave} if for any point $x\in \Omega$ there is a neighborhood $\Omega_x$ of $x$ and $\lambda \in \R$ such that the restriction $f|_{\Omega_x}$ is $\lambda$-concave.

\begin{prop}\label{prop:semiconcave}
If the support of $\rho$ is bounded (as in has bounded diameter) then the corresponding Fr\'echet function is semiconcave.
\end{prop}
\begin{proof}
We will first show that if the support of a  probability distribution $\rho$ is bounded then the corresponding Fr\'echet function is Lipschitz on any set with bounded diameter. We then show that for any unit length geodesic $\gamma$ and any $X \in \D_{L^2}$ the function $$g_X(s):=d(\gamma(s), X)^2 - s^2$$ is concave. We then complete the proof by showing the Fr\'echet function $F$ is 2-concave at every point (and hence $F$ is semiconcave) by considering $F(\gamma(s))-s^2$ as $\int g_X(s) d\rho(X)$. 

Let $U$ be a subset of $\D_{L^2}$  with bounded diameter. This means that there is some $K$ such that for any $Y \in U$ we have $\int d(X,Y) d\rho(X) \leq K$. Here we are also using that the support of $\rho$ is bounded. Let $Y,Z \in U$. Then
\begin{align*}
|F(Y)-F(Z)|&=\left|\int d(X,Y)^2-d(X,Z)^2 d\rho(X)\right|\\
&= \left|\int (d(X,Y) - d(X,Z))(d(X,Z)+ d(X,Y)) d\rho(X)\right|\\
&\leq \int (d(Z,Y))(d(X,Z)+ d(X,Y))d\rho(X).\\
&=2Kd(Z,Y).
\end{align*}
 
Let $\gamma$ be a unit speed geodesic and $X\in \D_{L^2}$. Consider the function 
$$g_X(s):=d(\gamma(s), X)^2 - s^2.$$ We want to show that $g_X$ is concave which means that $g_X(tx+(1-t)y)\geq tg_X(x) + (1-t)g_X(y)$.  Let $\tilde\gamma(t)$ be the geodesic from $\gamma(x)$ to $\gamma(y)$ traveling along $\gamma$ so that $\gamma((1-t)x+ ty) = \tilde\gamma(t)$ for $t \in [0,1]$ and
\begin{align*}
t g_X(x)+(1-t)g_X(y)&= t d(\tilde{\gamma}(0),X)^2 + (1-t) d(\tilde{\gamma}(1),X)^2 -tx^2-(1-t)y^2\\
&\leq d(\tilde{\gamma}(t),X)^2 +t(1-t) d (\tilde{\gamma}(0), \tilde{\gamma}(1))^2  -tx^2-(1-t)y^2\\
&=d(\tilde{\gamma}(t),X)^2 +t(1-t)(x-y)^2  -tx^2-(1-t)y^2\\
&=d(\tilde{\gamma}(t),X)^2 -(tx+(1-t)y)^2\\
&=g_X(tx+(1-t)y).
\end{align*}

The inequality comes from the defining inequality \eqref{eq:alex} that makes $\D_{L^2}$ a non-negatively curved Alexandrov space. 

By the construction of $g_X$ we can think of $F(\gamma(s))-s^2$ as $\int g_X(s) d\rho(X)$. This means that we can write
$$t[F(\gamma(x))-x^2]+ (1-t)[F(\gamma(y))-y^2] = \int tg_X(x)+(1-t)g_X(y) d\rho(X).$$ The concavity of $g_X$ ensures that $ tg_X(x)+(1-t)g_X(y)\leq g_X(tx+(1-t)y)$ and hence
\begin{align*}
t[F(\gamma(x))-x^2]+ (1-t)[F(\gamma(y))-y^2]&\leq \int g_X(tx+(1-t)y) d\rho(X)\\
&=F(tx+(1-t)y)  -(tx+(1-t)y)^2
\end{align*}
\end{proof}

We now define the additional structure on Alexandrov spaces with curvature bounded from below that we will need to define
gradients and supporting vectors. This exposition is a summary of the content in \cite{Ohta,Petrunin}. 

Given a point $Y$ in  an Alexandrov space $\mathcal{A}$  with non-negative curvature we first define the 
tangent cone $T_Y$. Let $\widehat{\Sigma}_Y$ be the set of all nontrivial unit-speed geodesics 
emanating from $Y$. For $\gamma, \eta \in \widehat{\Sigma}_Y$
the angle between them defined by
$$\angle_Y(\gamma, \eta) :=\arccos  \left(\lim\limits_{s,t\downarrow0}\frac{s^2+t^2-d(\gamma(s),\eta(t))^2}{2st}\right)\in [0,\pi],$$
when the limit exists. We define the space of directions $(\Sigma_Y, \angle_Y)$ at $Y$ as the completion of $\widehat{\Sigma}_Y/ \sim$ with respect to $\angle_Y$, where $\gamma
\sim \eta$ if $\angle_Y(\gamma,\eta)=0$. The \emph{tangent cone} 
$T_Y$ is the Euclidean cone over $\Sigma_Y$:
\begin{align*}
T_Y &:= \Sigma_Y \times [0,\infty)/\Sigma_Y\times\{0\}\\
d_{T_Y}((\gamma, s), (\eta, t))^2&:= s^2+t^2 - 2st\cos \angle_Y(\gamma, \eta).
\end{align*}
The inner product of $\mathbf{u} = (\gamma,s), \mathbf{v} = (\eta,t)
\in T_Y$ is defined as
$$\langle \mathbf{u}, \mathbf{v} \rangle_Y := st \cos
\angle_Y(\gamma,\eta) = \frac{1}{2}\left[ s^2 + t^2 -
  d_{T_Y}(\mathbf{u},\mathbf{v})^2 \right].$$ 

A geometric description of the tangent cone $T_Y$ is as follows. $Y \in \D_{L^2}$ has countably many points $\{y_i\}$ off the diagonal. A tangent vector is a set of vectors $\{v_i \in \R^2\}$ one assigned to each $y_i$ along with countably many vectors at points along the diagonal pointing perpendicular to the diagonal such that the sum of the squares of the lengths of all these vectors is finite. Observe that there can exist tangent vectors such that the corresponding  geodesic may not exist for any positive amount of time. The angle between two tangent vectors is effectively a weighted average of all the angles between the pairs of vectors.

We now define differential structure as a limit of rescalings. For $s > 0$ denote the space $(\mathcal{A},s\cdot d)$ by $s \mathcal{A}$ and
define the map $i_s: s \mathcal{A} \rightarrow \mathcal{A}$. For an open set $\Omega \subset \mathcal{A}$ and  any function $f : \Omega \rightarrow \R$ the
differential of $f$ at a point $p \in \Omega$ is a map $T_p \rightarrow \R$  is defined by
$$d_p f = \lim_{s \rightarrow \infty} s (f \circ i_s - f(p)), \quad f \circ i_s: s \mathcal{A} \rightarrow \R.$$

For semiconcave functions the above differential is well defined and we can study gradients and supporting
vectors. 

\begin{definition}[Gradients and supporting vectors]
Given an open set $\Omega \subset \mathcal{A}$ and a function $f:  \Omega \rightarrow \R$ we denote by $\nabla_p f$  the \emph{gradient} of a function 
$f$ at a point $p \in \Omega$.  $\nabla_p f$ is the vector $v \in T_p$ such that
\begin{enumerate}
\item[(i)] $d_p f(x) \leq \langle v, x\rangle$ for all $x\in T_p$
\item[(ii)] $d_p f(v)=\langle v,v\rangle$.
\end{enumerate}
For a semiconcave $f$ the gradient exists and is unique (Theorem 1.7 in \cite{Lyt}).
We say $s\in T_p$ is a \emph{supporting vector} of $f$ at $p$ if  
$d_p f(x) \leq - \langle s, x\rangle$ for all $x\in T_p$.
 Note that 
$-\nabla_p f$ is a supporting vector if it exists in the tangent cone at $p$.
\end{definition}

\begin{lem}\label{lem:suppvectoriszero}
\begin{enumerate}
\item[(i)] If  $s$ is a supporting vector then $\|s\| \geq \|\nabla_p f\|$.
\item[(ii)] If $p$ is local minimum of $f$ and $s$ is a supporting vector of $f$ at $p$ then $s=0$.
\end{enumerate}
\end{lem}
\begin{proof}
(i) First observe that  from the definitions of $\nabla_p f$ and supporting vectors we have 
$$\langle  \nabla_p f, \nabla_p f \rangle= d_p f(\nabla_p f)\leq -\langle s, \nabla_p f\rangle.$$
We also know that 
$$0\leq \langle \nabla_p f +s, \nabla_p f+s\rangle= \langle \nabla_p f, \nabla_p f\rangle + 2\langle \nabla_p f,s\rangle + \langle s, s\rangle.$$ These inequalities combined tell us that $0\leq -\langle \nabla_p f, \nabla_p f\rangle +  \langle s, s\rangle.$

(ii) If $p$ is a local minimum of $f$ then $d_pf(x)\geq 0$ for all $x\in T_p$. In particular $d_p(s)\geq 0$. Since $s$ is a supporting vector $-\langle s, s \rangle \geq d_p f(s) \geq 0$. This implies $\langle s,s\rangle =0$  and hence $s=0$.
\end{proof}

We care about gradients and supporting vectors because they can help
us find local minima of the Fr\'echet function. Indeed a necessary 
condition for $F$ to have local minimum at $Y$ is $s=0$ for any supporting vector $s$ of $F$ at $Y$. Since the tangent cone at $Y$ is a convex subset of a Hilbert space we can take integrals over probability measures with values in $T_Y$. This allows us to find a formula for a supporting vector of the Fr\'echet function $F$.

\begin{prop}\label{lem:sumsupp}
Let $Y\in \D_{L^2}$. For each $X\in \D_{L^2}$ let $F_X:Z \mapsto d(X, Z)^2$.
\begin{enumerate}
\item[(i)]  If $\gamma$ is a distance achieving geodesic from $Y$ to $X$, then the tangent vector to $\gamma$ at $Y$ of length $2d(X,Y)$ is a supporting vector at $Y$ for $F_X$.
\item[(ii)] If $s_X$  is a supporting vector at $Y$ for the function $F_X$ for each $X\in \text{supp}(\rho)$ then $s=\int s_Xd\rho(X)$ is a supporting vector at Y of the Fr\'echet function $F$ corresponding to the distribution $\rho$. 
\end{enumerate}
\end{prop}
\begin{proof}
(i) Let $\gamma$ be a unit speed geodesic from $Y$ to $X$. Consider the tangent vector $s_X=(\gamma, 2d(X,Y))$. Let $\gamma(t)_i$ denote the point in $\gamma(t)$ that is sent to $x_i \in X$. Since $\gamma$ is a distance achieving geodesic we know that 
$$\inf_{\phi:\gamma(0)\to X} \sum_i\|x_i -\phi(x_i)\|^2 = \sum_i\|x_i-\gamma(0)_i\|^2=F_X(Y).$$

To show $d_YF_X(v)\leq \langle s_X, v\rangle$ for all $v\in T_Y$ it is sufficient to consider vectors of the form $(\tilde{\gamma}, 1)$ where $\tilde{\gamma}$ is a unit speed geodesic starting at $Y$. Let $\tilde{\gamma}(t)_i$ denote the point in $\tilde{\gamma}(t)$ which started at $\gamma(0)_i$. This means that $x_i \mapsto \tilde{\gamma}(t)_i$ is a bijection from $X$ to $\tilde{\gamma}(t)$ and
\begin{align*}
d_Y F_X(v) &= \frac{d}{dt}\bigg|_{t=0} F_X(\tilde{\gamma}(t))\\
&=\lim_{t\to 0} \frac{F_X(\tilde{\gamma}(t))-F_X(Y)}{t}\\
&=\lim_{t\to 0}\frac{\inf\{\sum \|x_i-\phi(x_i)\|^2-\|x_i-\gamma(0)_i\|^2\,|\, \phi:X \to \tilde{\gamma}(t)\}}{t}\\
&\le \lim_{t\to 0}\frac{\sum \|x_i-\tilde{\gamma}(t)_i\|^2-\|x_i-\gamma(0)_i\|^2}{t}\\
&=\lim_{t\to 0}\frac{\sum \|\tilde{\gamma}(0)_i-\tilde{\gamma}(t)_i\|^2-2 \|\tilde{\gamma}(0)_i-\tilde{\gamma}(t)_i\| \|x_i-\gamma(0)_i\| \cos \theta_i}{t}\\
\end{align*}
where $\theta_i$ is the angle between the paths $s\mapsto \gamma(s)_i$ and $t\mapsto \tilde{\gamma}(t)_i$ in the plane. Now $$\|x_i-\gamma(0)_i\|=\|\gamma(d(X,Y))_i-\gamma(0)_i\|=d(X,Y)\frac{\|\gamma(s)_i-\gamma(0)_i\|}{s} $$ for all $s>0$ and $\|\tilde{\gamma}(0)_i-\tilde{\gamma}(t)_i\|^2=t^2\|\tilde{\gamma}(0)_i-\tilde{\gamma}(1)_i\|^2$ for all $t$. This implies that 
$$d_Y F_X(v) \leq -2d(X,Y)\lim_{t,s\downarrow 0}\frac{\sum \|\tilde{\gamma}(t)_i-\tilde{\gamma}(0)_i\| \|\gamma(s)_i-\gamma(0)_i\| \cos \theta_i}{st}.$$
Recall from our construction of the tangent cone that
\begin{align*}
\langle v, s_X\rangle&= 2d_{L^2}(X,Y) \cos(\angle_Y (\gamma, \tilde{\gamma}))\\
&=2 d(X,Y) \left(\lim_{s,t\downarrow 0} \frac{s^2+t^2 - d(\gamma(s), \tilde{\gamma}(t))^2}{2st}\right)\\
&=2d(X,Y) \left(\lim_{s,t\downarrow 0}\frac{\sum\|\gamma(s)_i-\gamma(0)_i\|^2+\|\tilde{\gamma}(t)_i-\tilde{\gamma}(0)\|^2 - \|\gamma(s)_i- \tilde{\gamma}(t)_i\|^2}{2st}\right)\\
&=2d(X,Y)\left(\lim_{t,s\downarrow 0}\frac{\sum \|\tilde{\gamma}(t)_i-\tilde{\gamma}(0)_i\| \|\gamma(s)_i-\gamma(0)_i\| \cos \theta_i}{st}\right).
\end{align*}
By comparing these equations we get $d_Y F_X(v) \leq -\langle v, s_X\rangle$ and thus we can conclude $s_X$ is a supporting vector.

(ii) Now let $s_X$ be any supporting vector of $F_X$. By its definition we know that $d_Y F_X(v) \leq - \langle s_X, v\rangle$ for all $v\in T_Y$ and hence 
\begin{align*}
d_Y F(v)&=\int d_Y F_X(v) d\rho(X) \leq \int \left( - \langle s_X, v  \rangle \right)d\rho(X)= -\left \langle \int  s_X d\rho(X) , v \right \rangle.
\end{align*}
\end{proof}

In the following section we provide an algorithm that computes a
local minimum of a Fr\'echet function using a gradient
descent procedure. The above results will be used since computing a supporting 
vector of $Z \mapsto d(X,Z)^2$ can be significantly 
easier and faster than computing a supporting vector of $F$ itself
 
\section{Finding local minima of the Fr\'echet function}
In this section we state an algorithm that computes a Fr\'echet mean 
of a finite set of persistence diagrams with finitely many off diagonal points, and examine convergence properties
of this algorithm. We will restrict our attention to diagrams with
only finitely many off-diagonal points with multiplicity of the points
allowed. 

Given a set of persistence diagrams $\{X_i\}_{i=1}^m$ a 
Fr\'echet mean $Y$ is a diagram that satisfies
$$\min_{Y \in \D_{L^2}} \left[ F_m := \int_{\D_{L^2}} d(X,Y)^2 d \rho_m (X) \right],$$
with the empirical measure $\rho_m := m^{-1} \sum_{i=1}^m \delta_{X_i}$.

We employ a greedy search algorithm based on gradient descent 
to find a local minimum. A key component of this greedy algorithm 
(see Algorithm \ref{findlocalmin}) consists of a variant of the 
Kuhn-Munkres (Hungarian) algorithm
\cite{Munkres57}.

The Hungarian algorithm finds the least cost assignment of tasks to people under the assumption that the number of tasks and people are the same. The input is the cost for each person to do each of the tasks. Suppose we have two diagrams $X$ and $Y$ each with only finitely many off diagonal points. Consider as many copies of the diagonal in $X$ and $Y$ to allow the option of matching every off diagonal point with the diagonal. We can think of the  points and copies of the diagonal in $X$ as the people and the points and copies of the diagonal in $Y$ as tasks. The cost of $x\in X$ doing task $y\in Y$ is $\|x-y\|^2$. The total cost of an assignment (or in other words bijection) $\phi$ of tasks to people is $\sum_{x\in X} \|x-\phi(x)\|^2$. The  Hungarian algorithm gives us a bijection $\phi$ that minimizes this cost. This means it gives an optimal pairing between $X$ and $Y.$

We would like to use the arithmetic mean of points in the plane and some number of copies of the diagonal. If $x_1, \ldots, x_m$ are points in $\R^2$ then there arithmetic mean $w=\frac{1}{n}\sum_{i=1}^m x_i$ is the choice of $z$ that minimizes the sum $\sum_{i=1}^m \|z-x_i\|^2$. If $x_i=\Delta$ for all $i$ then the arithmetic mean is set to be $\Delta$. The final case, without loss of generality, is when $x_1, \ldots , x_k$ are all off diagonal points and $x_{k+1}, \ldots, x_m$ are all the diagonal. Let $w$ be the normal arithmetic mean of $x_1, \ldots , x_k$ and let $w_\Delta$ be the closest point on the diagonal to $w$. We set 
$$w':=\frac{kw+(m-k)w_\Delta}{m}$$ to be the arithmetic mean of $x_1, \ldots, x_m$. This is the choice of $z$ that minimizes $\sum_{i=1}^m \|z-x_i\|^2$. We use an operation 
$\mbox{mean}_{i=1,..,m}(x_i^j)$ that computes the arithmetic mean
for each pairing over the diagrams. 

Suppose $Y$ is our current estimate for the Fr\'echet mean. Using the Hungarian algorithm we compute optimal pairings between $Y$ and each  of the $X_i$. We denote these pairings as $\{(y^j, x_i^j)\}_{j=1}^{J_i}$ where $J_i$ is the number of off diagonal in $X_i$ and $Y$ combined. For each $y_j\neq \Delta$ we then consider all the $x_{ij}$. Let $\tilde{y^j}$ be the arithmetic mean of the $x_{ij}$. Whenever in our pairings $\{(y^j, x_i^j)\}_{j=1}^{J_i}$ we see a $(\Delta, x_i^j)$ we think this as a different copy of the diagonal as in any pairing between $Y$ and $X_k$ with $k\neq i$. We would be using the arithmetic mean of $m-1$ copies of the diagonal and $x_i^j$. Let $Y'$ be the diagram with points $\tilde{y^j}$. We will show later that if $Y=Y'$ then $Y$ is a local minimum of the Fr\'echet function. Otherwise we chose $Y'$ to be our current estimate.

The basic steps of Algorithm 
\ref{findlocalmin} is to:
\begin{enumerate}
\item[(a)] randomly initialize the mean diagram. For example we can start at one of the $m$
persistence diagrams or the midway point of two of the $m$ diagrams;
\item[(b)] use the Hungarian algorithm to compute optimal pairings
between the estimate of the mean diagram and each of the persistence
diagrams;
\item[(c)] update each point in the mean diagram estimate with the 
arithmetic mean over all diagrams -- each point in the mean diagram
is paired with a point (possibly on the diagonal) in each diagram;
\item[(d)] if the updated estimate locally minimizes $F_m$
then return the estimate otherwise return to step (b). 
\end{enumerate}

\begin{algorithm}\label{findlocalmin}
\SetKwInOut{Input}{input} \SetKwInOut{Return}{return}

\vspace{.2in}

\Input{ \mbox{persistence diagrams} $\{X_1, \ldots, X_m\}$} \vspace{.1in}
\vspace{.1in}
 \Return{\mbox{Fr\'echet mean} $\{Y\}$} \vspace{.1in}

\vspace{.2in}

Draw $i \sim \mbox{Uniform}(1,...,n)$;  \tcc{randomly draw a diagram} 
Initialize $Y \leftarrow X_i$;  \tcc{initialize $Y$}

\vspace{.1in}
 stop $\leftarrow$ false \;

\Repeat{stop=true}{

$K = |Y|$;         \tcc{the number of non-diagonal points in $Y$} 

\For{i=1,\ldots, m}{
     $(y^j, x_i^j) \leftarrow \mbox{Hungarian}(Y,X_i)$ ; \tcc{compute optimal pairings between each $X_i$ and $Y$ using the Hungarian algorithm}   
     }
     
\For{j=1,\ldots K}{
    $y^j \leftarrow \mbox{mean}_{i=1,..,m}(x_i^j)$ \tcc{set each non-diagonal point in $Y$ to the arithmetic mean of its pairings} 
     }
     
\lIf{Hungarian$(Y,X_i) = (y_j,x_i^j)$}{stop $\leftarrow$ true} \tcc{The points in the updated $Y$ are optimal pairings w.r.t. each $X_i$}

}

\Return{Y}

\caption{Algorithm for computing the Fr\'echet mean $Y$ from persistence diagrams $X_1,\ldots, X_m$.}
\end{algorithm}

An alternative to the above greedy approach would be a brute force 
search over point configurations to find a Fr\'echet mean. One
way to do this is to list all possible pairings between points
in each pair of diagrams. Then compute the arithmetic mean for 
all such pairings. One of these means will be a Fr\'echet mean.
While this approach will find the complete mean set its combinatorial complexity is prohibitive.

\subsection{Convergence of the greedy algorithm}

The remainder of this section provides convergence properties for
Algorithm \ref{findlocalmin}. By convergence we mean that the algorithm
will terminate at some point having found a local minimum. The reason for this is
that at each iteration the cost function $F_m$ decreases, at each iteration the algorithm uses a new set of pairings, and there are
only finitely many combinations of pairings between points in the diagrams.

We first develop necessary and sufficient conditions for a diagram $Y$
to be a local minimum of a set of persistence diagrams. We define
$F_i (Z):=  d(Z,X_i)^2$, the Fr\'echet function corresponding to 
$\delta_{X_i}$. This allows us to define the Fr\'echet function
as $F= \frac{1}{m} \sum_{i=1}^m F_i$ corresponding to the the distribution $\frac{1}{m}\sum_{i=1}^m \delta_{X_i}$.

The following lemma provides a necessary condition for a diagram to be
a local minimum of $F$. This condition is the stopping
criterion in Algorithm \ref{findlocalmin}.

\begin{lem}\label{lem:necessary}
If $W = \{w_i\}$ is a local minimum of the Fr\'echet function $F = \frac{1}{m} \sum_{j=1}^m F_j$ $F$ then there is a unique optimal 
pairing from $W$ to each of the $X_j$ which we denote as $\phi_j$ and
each $w_i$ is the arithmetic mean of the points 
$\{\phi_j(w_i)\}_{j=1, 2 \ldots m}$. Furthermore if $w_k$ and $w_l$ are off-diagonal points such that $\|w_k-w_l\|=0$ then $\|\phi_j(w_k)-\phi_j(w_l)\|=0$ for each $j$.
\end{lem}

\begin{proof}
Let $\phi_j$ be some optimal pairings (not yet assumed to be unique)
between $Y$ and $X_j$ and let $s_j$ be the corresponding vectors in
the tangent cone at $Y$ that are tangent to the geodesics induced by
$\phi_j$ and are of length $d(X_j,Y)$. The $2s_j$ are supporting vectors for the functions $F_j(Y)=
 d(Y, X_j)^2$ by Proposition \ref{lem:sumsupp}, so we have $\frac2m\sum_{j=1}^m s_j$ is a supporting vector of $F$.

From Lemma \ref{lem:suppvectoriszero} we know that $\frac{2}{m}\sum_{j=1}^m s_j=0$. Since at each $w_i$ the $s_j$ gives the vector from $w_i$ to $\phi_j(w_i)$, $\sum_{j=1}^m s_j=0$ implies that $w_i$ is the arithmetic mean of the points $\{\phi_j(w_i)\}_{j=1, 2 \ldots m}$.

Now suppose that $\phi_k$ and $\tilde{\phi_k}$ are both optimal pairings. By the above reasoning we have 
$\frac1m(\tilde{s_k} + \sum_{j=1, j\neq k}^m s_j)=0 =\frac1m\sum_{j=1}^m s_j$  and hence $\tilde{s_k} = s_k$. This implies that $\|\tilde{\phi_k}(w_i) -\phi_k(w_i)\|=0$ for all $w_i\in W$. In particular, for off-diagonal points $w_k$ and $w_l$ with $\|w_k-w_l\|=0$ and $\phi_k$ an optimal pairing, we can consider the pairing $\tilde{\phi}_k$ with $w_k$ and $w_l$ swapped. Since $\|\tilde{\phi_k}(w_i) -\phi_k(w_i)\|=0$ for all $w_i\in W$ we can conclude that $\|\phi_j(w_k)-\phi_j(w_l)\|$.

\end{proof} 

We now prove that the above is also a sufficient condition for $W$ to be 
a local minimum of $F$ when $F$ is the Fr\'echet function for the measure $\frac{1}{m}\sum_i\delta_{X_i}$ withe the diagrams $X_i$ each with finitely many off-diagonal points. This requires a result about a local extension of optimal pairings.

\begin{prop}\label{prop:unique pairing}
Let $X$ and $Y$ be diagrams, each with only finitely many off diagonal points, such that there is a unique optimal pairing $\phi_X^Y$ between them and no off diagonal point in $X$ matches the diagonal in $Y$. We further stipulate that if $y_k$ and $y_l$ are off-diagonal points with $\|y_k-y_l\|=0$ then $\|(\phi_X^Y)^{-1}(y_k)-(\phi_X^Y)^{-1}(y_l)\|=0$. There is some $r>0$ such that for every $Z \in B(Y,r)$ there is a unique optimal pairing between $X$ and $Z$ and this optimal pairing is induced from the one from $X$ to $Y$. By this we mean there is a unique optimal pairing $\phi_Y^Z$ from $Y$ to $Z$ and that the unique optimal pairing from $X$ to $Z$ is $\phi_Y^Z \circ \phi_X^Y$.

Furthermore, if $X_1, X_2, \ldots, X_m$ and $Y$ are diagrams with finitely many off-diagonal points such that there is a  unique optimal pairing $\phi_{X_i}^Y$ between $X_i$ and $Y$ for each $i$ with the same conditions as above, then there is some $r>0$ such that for every $Z \in B(Y,r)$ there is a unique optimal pairing between each $X_i$ and $Z$ and this optimal pairing is induced by the one from $X_i$ to $Y$.
\end{prop}

\begin{proof}
Since $Y$ has only finitely many off-diagonal points there is some $\epsilon>0$ such that for every diagram $Z$ with $d(Y,Z)<\epsilon$ there is a unique geodesic from $Y$ to $Z$. 

For each bijection $\phi$ of points in $X$ to points in $Y$, define the function $g_\phi$ between $X$ and points in $B(Y,\epsilon)$ by setting
\begin{align*}
g_\phi(X,Z):=\sum_{x \in X} \|x - \phi_Y^Z(\phi(x))\|^2 +  \sum_{\{z\in Z: (\phi_Y^Z)^{-1}(z)= \Delta\}}\|z-\Delta\|^2,
\end{align*}
where $\phi_Y^Z$ is the optimal pairing that comes from the unique geodesic from $Y$ to $Z$. First note that $g_\phi(X,Z) \leq \sum_{x\in X} \|x - \phi_Y^Z(\phi(x))\|^2 + d(Y,Z)^2$. Since there are only finitely many points in $X$ and $Y$ there is a bound $M$ on $ \|x - \phi(x)\| + \epsilon$. $M$ is a bound on $\|x - \phi_Y^Z(\phi(x))\|$ for all $x$ and all $\phi$. We also know $\|\phi_Y^Z(\phi(x))-\phi(x)\| \leq d(Y,Z)$ for all $x \in X$. Let $K$ be the number of off-diagonal points in diagrams $X$ and $Y$ combined. 
\begin{align*}
g_\phi(X,Z)&\leq \sum \|x_i - \phi_Y^Z(\phi(x_i))\|^2+ d(Y,Z)^2,\\
&\leq \sum_{x\in X} (\|x-\phi(x)\| + \|\phi(x)-\phi_Y^Z(\phi(x))\|)^2+ d(Y,Z)^2,\\
&\leq \sum_{x\in X} (\|x-\phi(x)\|^2+ \|\phi(x)-\phi_Y^Z(\phi(x))\|^2\\
&\qquad +2 \|x-\phi(x)\| \|\phi(x)-\phi_Y^Z(\phi(x))\|)+ d(Y,Z)^2, \\
&\leq g_\phi(X,Y) + 2d(Y,Z)^2 + 2M d(Y,Z) \, K.
\end{align*}
Similarly 
$$g_\phi(X,Y)\leq g_\phi(X,Z)+ 2d(Y,Z)^2  + 2MK d(Z,Y).$$

Let $\phi_X^Y$ be the optimal pairing from $X$ to $Y$ which is assumed to be unique in the statement of the proposition. Let $\hat{\phi}$ be another bijection of points in $X$ to points in $Y$. Since there are only finitely many off-diagonal points in $X$ and $Y$ there are only finitely many possible $\hat\phi$.
Set 
$$\beta:= \min_{\hat{\phi}\neq \phi_X^Y}\left \{ g_{\hat{\phi}}(X,Y)-g_{\phi_X^Y}(X,Y)\right\} =  \min_{\hat{\phi}\neq \phi_X^Y}\left\{ g_{\hat{\phi}}(X,Y)-d(X,Y)^2\right \}$$
which must be positive as $\phi_X^Y$ is uniquely optimal by assumption.

Choose $r>0$ such that $4r^2+4MKr<\beta$. Now suppose that $g_\phi(Z,X)\leq g_{\phi_X^Y}(Z,X)$ for some $Z \in B(Y,r)$. This will imply that
\begin{align*}
g_\phi(X,Y) &\leq g_\phi(X,Z) + 2d(Y,Z)^2  + 2MK\,d(Z,Y),\\
&\leq g_{\phi_X^Y}(X,Y) +  4d(Y,Z)^2  +4MK \,d(Y,Z),\\
&<g_{\phi_X^Y}(X,Z) +\beta,
\end{align*}
which contradicts our choice of $\beta$.

Now suppose $X_1, X_2, \ldots, X_m$ and $Y$ are diagrams with finitely many off diagonal points such that there is a  unique optimal pairing $\phi_{X_i}^Y$ between $X_i$ and $Y$ for each $i$. By the above argument there are some $r_1, r_2,\ldots r_m>0$ such that for each $i$ and for every $Z \in B(Y,r_i)$ there is a unique optimal pairing between each $X_i$ and $Z$ and this optimal pairing is induced by the one from $X_i$ to $Y$. Take $r=\min\{r_i\}$ which is positive.
\end{proof}

The following theorem states that Algorithm \ref{findlocalmin} will
find a local minimum on termination.

\begin{thm}\label{localminres}
Given diagrams $\{X_1,...X_m\}$ and the corresponding Fr\'echet
function $F$, then $W = \{w_i\}$ is a local minimum of $F$ if and only
if there is a unique optimal pairing from $W$ to each of the $X_j$
denoted as $\phi_j$ and each $w_i$ is the arithmetic mean of the 
points $\{\phi_j(w_i)\}_{j=1, 2 \ldots m}$.
\end{thm}

\begin{proof}
In Lemma \ref{lem:necessary} we showed that it it is a necessary condition. 

Given $m$ points in the plane or copies of the diagonal, $\{x_1, x_2, \ldots, x_m\}$, the choice of $y$ which minimizes $\sum_{i=1}^m \|x_i-y\|^2$ is the arithmetic mean of $\{x_1, \ldots, x_m\}$. As a result we know that $F(Z)>F(W)$ for all $Z$ with the same optimal pairings as $W$ to $X_1, X_2, \ldots, X_m$. Since there is some ball $B(W,r)$ such that every $Z\in B(W,r)$ has the same optimal pairings as $W$, by proposition \ref{prop:unique pairing}, we know that $F(Z)>F(W)$ for all $Z$ in $B(W,r)$. Thus we can conclude that $W$ is a local minimum.
\end{proof}

\section{Law of large numbers for the empirical Fr\'echet mean}

In this section we study the convergence of Fr\'echet means
computed from sampling sets to the set of means of a measure. 
Consider a measure $\rho$ on the space of persistence
diagrams $\D_{L^2}$. Given a set of persistence diagrams $\{X_i\}_{i=1}^n
\stackrel{iid}{\sim} \rho$ one can define an empirical
measure $\rho_n=\frac{1}{n}\sum_{k=1}^n \delta_{X_k}$. We will
examine the relation between the two sets
\begin{eqnarray*}
\mathbf{Y} &=& \left\{ \min_{Z \in \D_{L^2}} \left[ F := \int_{\D_{L^2}} d(X,Z)^2 d\rho(X) \right] \right\} , \\
\mathbf{Y}_n &=& \left\{ \min_{Z \in \D_{L^2}} \left[ F_n := \int_{\D_{L^2}} d(X,Z)^2 d\rho_n(X) \right] \right\},
\end{eqnarray*}
where $\mathbf{Y}$ and $\mathbf{Y}_n$ are the Fr\'echet mean sets of 
the measures $\rho$ and $\rho_n$ respectively. We would like
prove convergence of $\mathbf{Y}_n$ to $\mathbf{Y}$ asymptotically
with $n$ -- a law of large numbers result.

There exist weak and strong laws of large numbers for general metric
spaces (for example see
\cite{Molchanov}[Theorem 3.4]). These results hold for global minima
of the Fr\'echet and empirical Fr\'echet functions $F$ and $F_n$, respectively.
It is not clear to us how to adapt these results to the
case of Algorithm  \ref{findlocalmin} where we can only ensure convergence
to a local minimum. It is also not clear how we can adapt these
theorems to get rates of convergence of the sample Fr\'echet mean
set to the population quantity.

In this section we provide a law of large number result for the
restricted case where $\rho$ is a combination of Dirac masses
$$\rho=\frac1m\sum_{i=1}^m \delta_{Z_i},$$
where $Z_i$ are diagrams with only finitely many off diagonal points 
and we allow for multiplicity in these points. The proof is
constructive and we provide rates of convergence.

The main results of this section, Theorem \ref{thm:ifflocal}
and Lemma \ref{finmin}, provide a probabilistic justification
for Algorithm \ref{findlocalmin}. Theorem \ref{thm:ifflocal}
states that with high probability local minima of
the empirical Fr\'echet function $F_n$ will be close to
local minima of the Fr\'echet function $F$.
Ideally we would like the above convergence to hold for global minima, the 
Fr\'echet mean set. The condition of Lemma \ref{finmin} states that
the number of local minima of $F_n$ is finite and not a function
of $n$. This suggests that applying Algorithm \ref{findlocalmin} to a random set of start conditions can be used to explore
the finite set of local minima.

 \begin{thm}\label{thm:ifflocal}
Set $\rho=\frac1m\sum_{i=1}^m \delta_{Z_i}$ where $Z_i$ are diagrams 
with finitely many off diagonal points with multiplicity allowed. Let
$F$ be the Fr\'echet function corresponding to $\rho$ and 
$Y$ be a local minimum of $F$. Set $\{X_i\}_{i=1}^n \stackrel{iid}{\sim} \rho$, 
and denote the corresponding empirical measure $\rho_n=\frac1n\sum_{k=1}^n
\delta_{X_k}$ and Fr\'echet mean function $F_n$. There exists a local minimum $Y_n$ of $F_n$ such that with probability
greater than $1-\delta$
$$d(Y,Y_n)^2 \leq  \frac{m^2 F(Y)}{n} \ln\left( \frac{m}{\delta}\right),$$
 for $n \geq 8m \ln \frac{m}{\delta}$ and  $\frac{m^2 F(Y)}{n} \ln\left( \frac{m}{\delta}\right) < r^2$
 where $r$ characterizes the separation between the local minima of $F$.
 \end{thm}
\begin{proof}
The empirical distribution is 
$$\rho_n=\frac1n\sum_{k=1}^n \delta_{X_k} = \frac1m\sum_{i=1}^{m} \xi_i \delta_{Z_i}$$ 
where $\xi_i$ is the random variable that states the multiplicity of each $Z_i$ appearing in the empirical measure, $|\{k:X_k=Z_i\}|$. 
Observe that $\xi_1, \xi_2, \ldots, \xi_m$ can be stated as a multinomial distribution with parameters $n$ and $p=\left(\frac1m,  \frac1m, \ldots, \frac1m\right)$.

We will bound the probability that $|\xi_i-\frac{n}{m}|>\epsilon\frac{n}{m}$ for any $i=1, 2, \ldots m$. We then will show that under the assumption that $|\xi_i-\frac{n}{m}|\le \epsilon\frac{n}{m}$ for all $i=1, 2, \ldots m$ for sufficiently small $\epsilon>0$ there is a local minimal $Y_n$ with $d(Y,Y_n)^2<\frac{\epsilon^2 m F(Y)}{(1-\epsilon)^2}$.

For each $i$, $\xi_i \sim \mbox{Bin}(n,\frac{1}{m})$ and  $n-\xi_i \sim \mbox{Bin}(n, 1-\frac{1}{m})$. Using Hoeffding's inequality we obtain
$\Pr\left[\xi_i-\frac{n}{m}\le -\epsilon\frac{n}{m}\right]\le \frac{1}{2}\exp(-2\frac{\epsilon^2 n}{m})$
and
\begin{align*}
\Pr\left[\xi_i-\frac{n}{m}\ge \epsilon\frac{n}{m}\right]=\Pr\left[(n-\xi_i)-(n-\frac{n}{m})\le-\epsilon\frac{n}{m}\right]\le \frac{1}{2}\exp\left(-2\frac{\epsilon^2 n}{m}\right)
\end{align*}
Together they show that $\Pr\left[|\xi_i-\frac{n}{m}|\ge \epsilon\frac{n}{m}\right] \le \exp(-2\frac{\epsilon^2 n}{m})$ implying the bound $$\Pr\left[|\xi_i-\frac{n}{m}|<\epsilon\frac{n}{m} \text{ for all } i=1, 2, \ldots, m\right] \ge 1-m\exp\left(-2\frac{\epsilon^2 n}{m}\right).$$

From now on we will assume that $|\xi_i-\frac{n}{m}|<\epsilon\frac{n}{m}$ for all $i=1, 2, \ldots, m$. Let us consider our algorithm for finding a local minimal of $F_n$ starting at the point $Y$. We first define some notation. We denote the
points in $Y$ by $\{y_j\}_{j=1}^J$. We denote by $z_i^j := \phi_Y^{Z_i}(y_j)$ the point in  $Z_i$ that $y_j$ is paired to in the (unique) optimal bijection between $Y$ and $Z_i$. Recall that the $z_i^j$ could be the diagonal but from our assumption that $Y$ is a local minimum no off diagonal point in any $Z_i$ is paired with the diagonal in $Y$.

Let $(a_i^j, b_i^j)$ be the coefficients of the vector from $y_j$ to $z_i^j$ in the basis of $\R^2$ given by $(\frac{1}{\sqrt{2}},\frac{1}{\sqrt{2}})$ and $(-\frac{1}{\sqrt{2}},\frac{1}{\sqrt{2}})$. This basis has the advantage that when $z_i^j$ is the diagonal then $a_i^j=0$ and $b_i^j=d(y_j,\Delta)$. From our assumption that $Y$ is a local minimum we know that $\sum_{i=l}^m a_i^j=0$ and $\sum_{i=l}^m b_i^j=0$ for all $j$ and 
$$F(Y)=\frac{1}{m}\sum_{j=1}^J\sum_{i=1}^m ((a_i^j)^2 +  (b_i^j)^2).$$

For the moment fix $j$. Without loss of generality reorder the $Z_i$ so that the first $k$ (with $1\leq k \leq m$) of the $z_i^j$ are off the diagonal and the remained are copies of the diagonal. Let $y_j^n$ be the point in $\R^2$ given by 
$$y_j + \left(\frac{1}{\xi_1 +\xi_2 + \ldots \xi_k} \sum_{i=1}^k \xi_i a_i\right) \left(\frac{1}{\sqrt{2}},\frac{1}{\sqrt{2}}\right)
+ \left(\frac{1}{n}\sum_{i=1}^m \xi_i b_i^j\right) \left(-\frac{1}{\sqrt{2}},\frac{1}{\sqrt{2}}\right).$$
By construction this $y_j^n$ is the weighted arithmetic mean of the $z_i^j$ where we have weighted by the $\xi_i$ taking into account that when $i>k$ then $z_i^j$ is the diagonal.

Under our assumption that $|\xi_i-\frac{n}{m}|<\epsilon\frac{n}{m}$ for all $i=1, 2, \ldots, m$ and using $\sum_{i=1}^k a_i^j=0=\sum_{i=1}^m b_i^j$ we know that
\begin{align*}
\|y_j-y_j^n\|^2 &= \frac{1}{(\xi_1 +\xi_2 + \ldots \xi_k)^2}\left( \sum_{i=1}^k \xi_i a_i^j\right)^2 + \frac{1}{n^2}\left(\sum_{i=1}^m \xi_ib_i^j\right)^2\\
&= \frac{1}{(\xi_1 +\xi_2 + \ldots \xi_k)^2}\left( \sum_{i=1}^k (\xi_i-\frac{n}{m}) a_i^j\right)^2 + \frac{1}{n^2}\left(\sum_{i=1}^m (\xi_i-\frac{n}{m})b_i^j\right)^2\\
&\leq \frac{1}{\frac{k^2}{m^2}n^2(1-\epsilon)^2}\frac{\epsilon^2n^2}{m^2}\left( \sum_{i=1}^k (a_i^j)^2\right) + \frac{1}{n^2}\frac{\epsilon^2n^2}{m^2}\left(\sum_{i=1}^m (b_i^j)^2\right)\\
&\le \frac{m\epsilon^2}{(1-\epsilon)^2}\left( \frac{1}{m} \sum_{i=1}^m (a_i^j)^2 + (b_i^j)^2\right).
\end{align*}

Set $Y_n$ to be the diagram with off-diagonal points $\{y_j^n\}_{j=1}^J$. Using the pairing between $Y$ and $Y_n$ where we pair $y_j$ with $y_j^n$ we conclude that
\begin{align*}
d(Y,Y_n)&\leq \sum_{j=1}^J \|y_j-y_j^n\|^2\\
& \leq  \sum_{j=1}^J \frac{m\epsilon^2}{(1-\epsilon)^2}\left( \frac{1}{m} \sum_{i=1}^m (a_i^j)^2 + (b_i^j)^2\right)\\
&\le \frac{m\epsilon^2}{(1-\epsilon)^2}F(Y).
\end{align*}
Set $\delta = m\exp\left(-2\frac{\epsilon^2 n}{m}\right)$ and solve for $\epsilon$. This provides the bound that with probability greater than $1-\delta$
$$d(Y,Y_n)^2 \leq  \frac{m^2 F(Y)}{ 2 n} \ln\left( \frac{m}{\delta}\right) \frac{1}{(1-\epsilon)^2}.$$
For $\epsilon \in [0,.25]$ it holds that $(1-\epsilon)^{-2} <2$ and $n \geq 8m \ln \frac{m}{\delta}$ implies $\epsilon < .25$.

We want to show that $Y_n$ is a local minimum for sufficiently small $\epsilon$. Indeed it will be the output of Algorithm \ref{findlocalmin} given the initializing diagram of $Y$. Since $Y$ is a local minimum, Proposition \ref{prop:unique pairing} implies that there is a ball around $Y$, $B(Y,r)$, such that for every diagram in $B(Y,r)$ there is a unique optimal pairing with each $Z_i$ which corresponds to the unique optimal pairing between $Y$ and $Z_i$. That is $\phi_X^{Z_i} = \phi_X^Y\circ \phi_Y^{Z_i}$ for all $X\in B(Y,r)$. For $\epsilon>0$ such that $\frac{\epsilon^2 m F(Y)}{(1-\epsilon)^2}<r^2$ we have $Y_n\in B(Y,r)$. Plugging in for $\epsilon$ results in
 $\frac{m^2 F(Y)}{n} \ln\left( \frac{m}{\delta}\right) < r^2$.

This implies that $\phi_{Y_n}^{Z_i} = \phi_{Y_n}^Y\circ \phi_Y^{Z_i}$ is the unique optimal pairing between $Y_n$ and $Z_i$ for all $i$ and hence $\phi_{Y_n}^{X_k} = \phi_{Y_n}^Y\circ \phi_Y^{X_k}$ for each of the sample diagrams $X_k$. If $X_k=Z_i$ then 
$$\phi_{Y_n}^{X_k}(y_j^n)=\phi_{Y_n}^Y\circ \phi_Y^{Z_i}(y_j^n) = \phi_Y^{Z_i}(y_j)=z_i^j.$$
By construction $y_j^n$ is the weighted arithmetic mean of the $z_i^j$ (weighted by the $\xi_i$), and hence $y_j^n$ is the arithmetic mean of the $x_k^j$. By Theorem \ref{localminres} $Y_n$ is local minimum.
\end{proof}
 
The above theorem provides a (weak) law of large numbers results for the local minima computed from $n$ persistence diagrams but it does not ensure that
the number of local minima is bounded as $n$ goes to infinity. The utility of such a convergence result would be limited if the number of local minima could not
be bounded. The following lemma states that the number of local minima is bounded.
  
\begin{lem}\label{finmin}
Let $\rho=\frac1m\sum_{i=1}^m \delta_{Z_i}$ as before. Let $\rho_n=\frac1n\sum_{k=1}^n \delta_{X_k}$ be the empirical measure of $n$ points drawn
iid from $\rho$ and $F_n$ is the corresponding Fr\'echet function. The number of local minima of $F_n$ is bounded by $\prod_{i=1}^m(k_i+1)^{(k_1+k_2+\ldots k_m)}$. Here $k_i$ is the number of off-diagonal points in the $i$-th diagram. This bound is independent of $n$. 
\end{lem}
\begin{proof}
Set $Y_n$ as a local minimum of $F_n$. This implies there are unique optimal pairings $\phi_i$ between  $Y_n$ and $X_i$ for each $i$ and that any point $y$ in $Y_n$ is the arithmetic mean  of $\{\phi_i(y)\}$. Since the optimal pairing is unique, if $X_i=X_j$ then $\phi_i=\phi_j$. This in turn means that the $\phi_i$ are determined by which of $Z_i$ are in the set $X_j$ (with multiplicity). This implies that the number of local minima is bounded by the number of different partitions into subsets of the points in the $\cup X_j$ so that each subset has exactly one point from each of the $X_j$.  The number of subsets is bounded by $k_1+k_2+\ldots +k_m$ and for each subset there is a bound of $\prod_{i=1}^m(k_i+1)$ on the choices of which element to take from each of the $X_i$.  Thus the number of different partitions is bounded by $\prod_{i=1}^m(k_i+1)^{(k_1+k_2+\ldots k_m)}$. 
\end{proof}

We would like to discuss not only the convergence of local minima but also the convergence of the Fr\'echet means. We can do this in the case when there is a unique Fr\'echet mean.

\begin{lem}
Let $\rho=\frac1m\sum_{i=1}^m \delta_{Z_i}$ as before. Suppose further that the corresponding Fr\'echet function $F$ has a unique minimum. Let $\rho_n=\frac1n\sum_{k=1}^n \delta_{X_k}$ be the empirical measure of $n$ points drawn
iid from $\rho$ and $F_n$ is the corresponding Fr\'echet function. Let $\mathbf{Y}$ be the Fr\'echet mean of $F$ and $\mathbf{Y}_n$ the set of Fr\'echet means of $F_n$. With probability $1$ the Hausdorff distance between $\mathbf{Y}_n$ and $\mathbf{Y}$ goes to zero as $n$ goes to infinity.
\end{lem}
\begin{proof}
It is sufficient for us to show for each $r>0$ that with probability $1$ there is some $N_r$ such that  $\mathbf{Y}_n\subset B( \mathbf{Y},r)$ for all $n>N_r$. 

Fix $r>0$. Suppose  there does not exist some $N_r$ such that  $\mathbf{Y}_n\subset B( \mathbf{Y},r)$ for all $n>N_r$. Then there is some sequence of $W_{n_k}\in \mathbf{Y}_{n_k}$ such that $d(W_{n_k}, \mathbf{Y})\ge r$. The set $\{W_{n_k}\}$ is clearly bounded, off-diagonally birth-death bounded and  uniform and hence precompact.  This implies that $(W_{n_k})$ has a convergent subsequence $(W_{{n_k}_j})$. Let $W$ denote the limit of this sequence. Since $d(W_{{n_k}_j}, \mathbf{Y})\ge r$ for all $j$ we have $d(W, \mathbf{Y})\ge r$.

By the arguments in Proposition \ref{prop:semiconcave} there is some $K$ independent of $n$ such that $F_n$ is $K$-Lipschitz in $B(W,1)$ and hence $|F_{{n_k}_j}(W_{{n_k}_j}) - F_{{n_k}_j}(W)|\leq K d(W_{{n_k}_j},W)$ for large $j$. Hence, for all $\epsilon>0$ we can say that $ F_{{n_k}_j}(W)\leq F_{{n_k}_j}(W_{{n_k}_j}) + \epsilon$ for sufficiently large $j$. 

The law of large numbers tells us that $F_n(W)\to F(W)$ and $F_n(\mathbf{Y})\to F(\mathbf{Y})$ as $n \to \infty$ with probability $1$. Hence for all $\epsilon>0$ we know that with probability $1$ both $F(W)\leq F_n(W) +\epsilon$ and  $F_n(\mathbf{Y})\leq F(\mathbf{Y}) +\epsilon$ for sufficiently large $n$.

From our assumption that $W_{{n_k}_j}$ is a Fr\'echet mean of $F_{{n_k}_j}$ we know that $F_{{n_k}_j}(W_{{n_k}_j})\leq F_{{n_k}_j}(\mathbf{Y})$ for all $j$.

Let $\epsilon>0$. Combining the inequalities above we conclude that with probability $1$ 
\begin{align*}
F(W)\leq  F_{{n_k}_j}(W) + \epsilon 
\leq  F_{{n_k}_j}(W_{{n_k}_j}) + 2\epsilon
\leq F_{{n_k}_j}(\mathbf{Y})+2\epsilon
\leq F(\mathbf{Y}) + 3\epsilon,
\end{align*}
for $j$ sufficiently large. Since $\epsilon>0$ was arbitrary we obtain $F(W)\leq F(\mathbf{Y})$ which contradicts the uniqueness assumption about the Fr\'echet mean.
\end{proof}

\section{Persistence diagrams of random Gaussian fields}
We illustrate the utility of our algorithm in computing means and
variances of persistence diagrams in this section via simulation.
The idea will be to show that persistence diagrams generated from
a random Gaussian field will concentrate around the diagonal with
the mean diagram moving closer to the diagonal as the number
of diagrams averaged increases.

The persistence diagrams were computed from random Gaussian field
over the unit square using the procedure outlined in Section
3 in \cite{ABBSW}. The field generated is a stationary, isotropic, and 
infinitely differentiable random field. The Gaussian was set to be
mean zero and the covariance function was $R(p) = \exp(-\alpha \|p
\|^2)$ where $\alpha = 100$. A few hundred levels in the range of
the realization of the field were taken for each level a simplicial
complex was constructed. This was done by taking a fine grid on 
the unit square and including any vertex, edge or square in the 
complex if and only if the values of the field at the vertex
or set of vertices (for the edge and square cases) were higher
than the level. The complex increases as the level decreases which
provides the filtering and from which birth-death values of the
diagram were computed. We obtained from E. Subag $10,000$ such random 
persistence diagrams generated as described above.
These diagrams contain points with infinite persistence, we ignore these points.
Using extended persistence in computing the diagrams would
address this issue.

\begin{figure}[hbt]
\begin{center}
\includegraphics[height=3.5in]{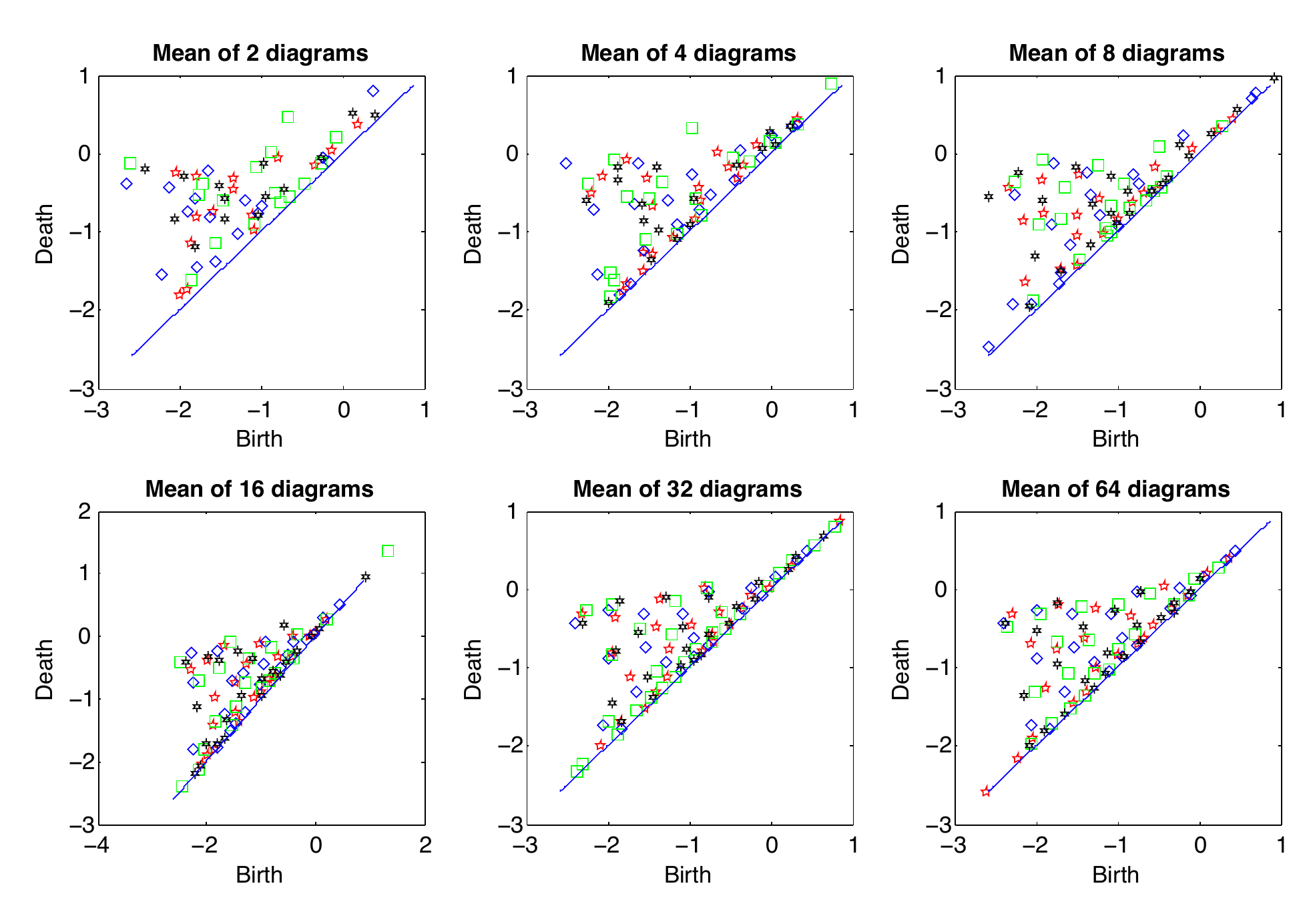}\\
\includegraphics[height=3.5in]{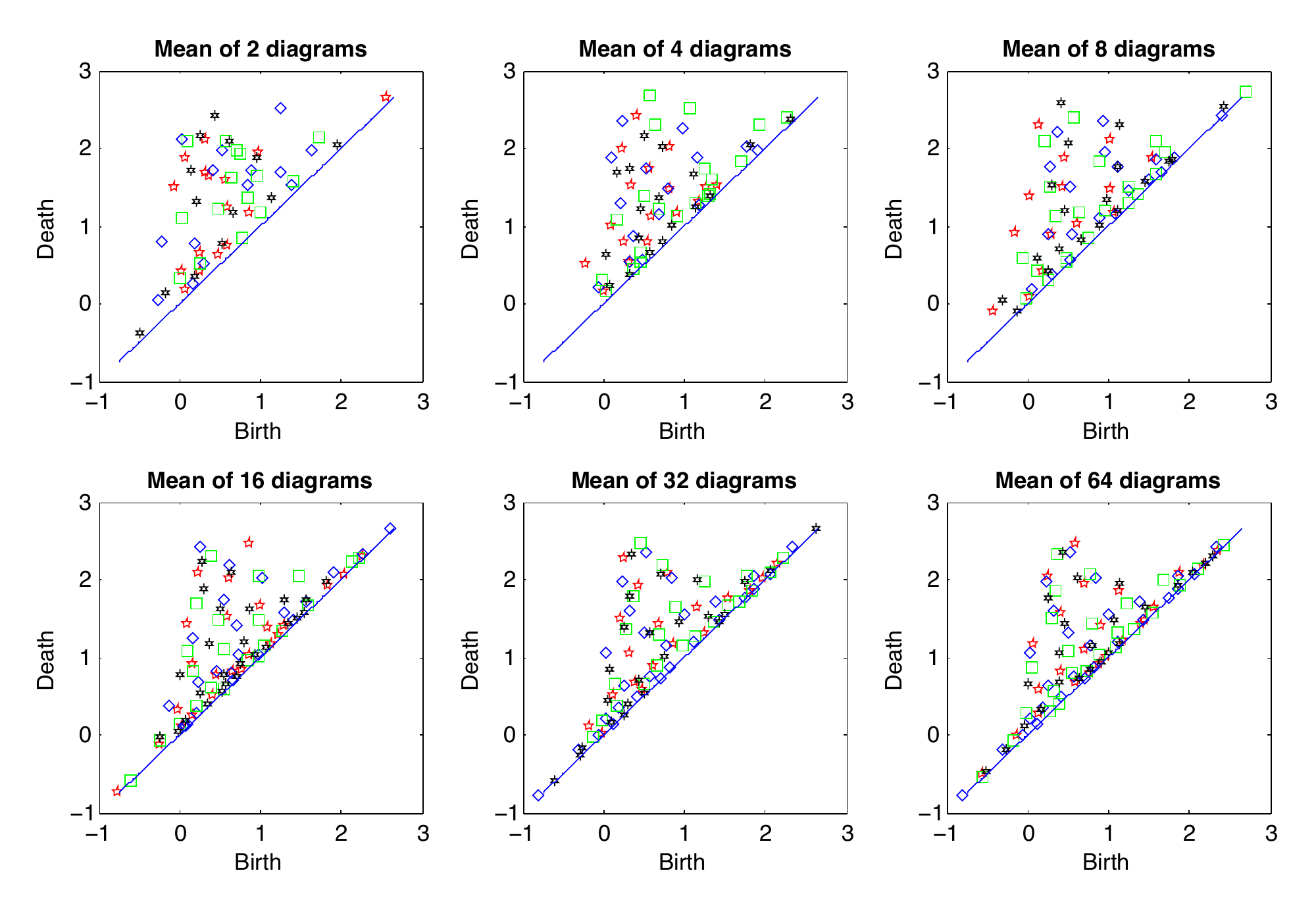} 
\caption{\label{concentration}
The top two rows plot the mean persistence diagram for dimension zero.
Each figure contains four means computed from the number of diagrams
specified in the figure title. Each mean is computed from a different
random sample of diagrams and is plotted in a different color. The 
bottom two rows are the sample plots for dimension one.
}
\end{center}
\end{figure}

\clearpage

In Figure \ref{concentration} we display the mean diagram of sets of
$2, 4, 8, 16, 32, 64, 128$ diagrams randomly drawn from the $10,000$
diagrams. This is done for both dimensions zero and one.
We wanted to see that as the number of diagrams being averaged 
increases the Fr\'echet means
converged. To quantify this concentration we took
ten draws of $2, 4, 8, 16, 32, 64, 128$ diagrams from the $10,000$ diagrams 
and considered the distribution $\frac{1}{10}\sum_{i=1}^{10}\delta_{X_i}$ where $X_i$ where the Fr\'echet means of each of the sets of samples. We then computed the variance of these distributions as documented in Table \ref{table:variance}. 
\begin{table}[ht]
\caption{Variance of the sample Fr\'echet Means}
\centering
\begin{tabular}{c c c c}
\hline\hline
Number of samples & $H_0$ & $H_1$  \\ [0.5ex] 
\hline
2& 0.8353 & 0.9058 \\
4& 0.6295 & 0.6741 \\
8& 0.4429 & 0.5608  \\
16 & 0.4356 & 0.4618 \\
32 & 0.3165 & 0.3742 \\ 
64 & 0.3362 & 0.2965\\
128 & 0.3127 & 0.2233\\[1ex]
\hline
\end{tabular}
\label{table:variance}
\end{table}

\section{Discussion}
In this paper we introduce an algorithm for computing estimates of Fr\'echet means of a set of persistence diagrams.
We demonstrate local convergence of this algorithm and provide a law of large numbers for the Fr\'echet mean
computed on this set when the underlying measure has the form $\rho = m^{-1} \sum_{i=1}^m \delta_{X_i}$,
where $X_i$ are persistence diagrams. We believe that generically there is a unique global minimum to the Fr\'echet function and hence a unique Fr\'echet mean but this needs to be shown.

The work in this paper is a first step and several obvious extensions 
are needed. A law of large numbers result when the underlying measure 
is not restricted to a combination of Dirac functions is obviously important.
The results in our paper are strongly dependent on the 
$L^2$-Wasserstein metric; generalizing these results
to the Wasserstein metrics used in computational topology is of
central interest. The proofs and problem formulation
in this paper are very constructive -- the proofs and algorithms are 
developed for the specific examples and constructions
we propose and are not meant to generalize to other metrics or 
variants on the algorithm. It would be of great interest to
provide a presentation of the core ideas in the algorithm and theory
we developed in a more general framework using properties of abstract 
metric spaces and probability theory on these spaces.

\section*{Acknowledgments}
SM and KT would like to acknowledge Shmuel Weinberger for discussions and
insight. SM and KT would like to acknowledge E. Subag with help in obtaining
persistence diagrams computed from random Gaussian fields and explaining the
generative model. JH and YM are pleased to acknowledge the support from grants
DTRA: HDTRA1-08-BRCWMD, DARPA: D12AP00001On, AFOSR: FA9550-10-1-0436,
and NIH (Systems Biology): 5P50-GM081883. SM is pleased to acknowledge
support from grants NIH (Systems Biology): 5P50-GM081883, AFOSR:
FA9550-10-1-0436, and NSF CCF-1049290.

\appendix
\section{}

In order to prove Proposition \ref{prop:infismin} we need to give some 
conditions for a subset of $\D_{L^2}$ to be relatively compact. We
will use Theorem 21 in \cite{MilMukHar:2012} which requires a few 
definitions.

\begin{definition}[Birth-death bounded]
A set $S \subset \D_{L^2}$ is called \emph{birth-death bounded}, if there is a
constant $C>0$ such that for all $Z \in S$ and  for all $\Delta\neq x\in Z$
$\max\{|\bb(x)|, |\dd(x)|\}\leq C$, where $\bb(x)$ and $\dd(x)$
are the births and deaths respectively. 
\end{definition}

For $\alpha >0 $ and diagram $Z \in \D_{L^2}$ we define the maps
\begin{eqnarray*}
u_\alpha: {\D_{L^2}} \rightarrow {\D_{L^2}} & \mbox{ such that } & \Delta\neq x \in
u_\alpha(Z) \Longleftrightarrow x \in  Z \, \& \, \mbox{pers}(x) \geq
\alpha \\
l_\alpha: {\D_{L^2}} \rightarrow {\D_{L^2}} & \mbox{ such that } & \Delta \neq  x \in
l_\alpha(Z) \Longleftrightarrow x \in  Z \, \& \,  \mbox{pers}(x) <
\alpha,
\end{eqnarray*}
where $u_\alpha(Z)$ is the $\alpha$-upper part of $Z$ (the points in
$Z$ with persistence
at least $\alpha$) and $l_\alpha(Z)$ is the $\alpha$-lower part of $Z$
(the points in $Z$
with persistence less than $\alpha$). 

\begin{definition}[Off-diagonally birth-death bounded]
A set $S\subset \D_{L^2}$ is called \emph{off-diagonally birth-death bounded} if for all $\epsilon > 0$, $u_\epsilon(S)$ is birth-death bounded.
\end{definition}

\begin{definition}[Uniform]
A set $S\subset \D_{L^2}$ is called \emph{uniform} if for all $\epsilon >0$ there exists $\alpha >0$ such
that $d(l_\alpha (Z), \Delta)\leq \epsilon$ for all $Z \in S$.
\end{definition}

Theorem 21 in \cite{MilMukHar:2012} states that a subset of $\D_{W_p}$
is relatively compact if and only if it is bounded, off-diagonally
birth-death bounded and uniform. This also holds for $\D_{L^2}$
due to the equivalence in norms stated in \eqref{normeq}. 
We finally are ready to prove  Proposition \ref{prop:infismin}. 

\begin{proof}[Proof of Proposition \ref{prop:infismin}]
Fix two diagrams $X$ and $Y$. Let $\Phi$ be the set of bijections $\phi$ between points in $X$ and points in $Y$ with the further condition that 
$$\|x - \phi(x)\|^2\leq \|x - \Delta\|^2 + \|\phi(x) - \Delta\|^2$$ 
for all $x\in X$. Recall that by $\|x -\Delta\|$ we mean the perpendicular distance from $x$ to the diagonal which can thought of as pairing $x$ with the closest point to $x$ on the diagonal. By the above condition we are requiring that we never pair an off diagonal point $x\in X$ with an off diagonal point in $Y$ when pairing both with the diagonal would be more efficient.

By considering only the bijections in $\Phi$ we are only removing bijections $\tilde{\phi}$ for which there exists some $\phi \in \Phi$ such that $\sum_{x\in X} \|x-\phi(x)\|^2 <\sum_{x\in X}\|x-\tilde{\phi}(x)\|^2$. This means that \eqref{eq:defdistance} is equal to $\inf \{\sum_{x\in X} \|x-\phi(x)\|^2:\phi \in \Phi\}$. We will show this infimum is a minimum.

For each bijection $\phi\in \Phi$ we can construct a path $\gamma_\phi:[0,1] \to \D_{L^2}$ by setting $\gamma_\phi(t)$ to be the diagram with points $\{(1-t)x_i+t\phi(x_i)|x_i\in X\}$. Let $S = \{\gamma_\phi (t) : t\in [0,1], \phi\in \Phi\}$ which contains all the images of the paths $\gamma_\phi$. We want to show that $S$ is relatively compact. To do this we will show that $S$ is bounded, off-diagonally birth-death bounded and uniform which are sufficient conditions for relative compactness by Theorem 21 in \cite{MilMukHar:2012}.

Firstly observe that for any bijection $\phi$ and any $t\in [0,1]$ we
know 
$$d(\gamma_\phi(t),\Delta)^2 \leq d(X, \Delta)^2 + d(Y,
\Delta)^2$$ 
which is finite and independent of $\phi$ and $t$. 
This implies that the set $S$ is bounded.

We now wish to show that $S$ is off-diagonally bounded. For each $\epsilon>0$ there can only be finitely many points in $X$ and $Y$ whose distance from the diagonal is at least $\epsilon$. This implies that there is some $\tilde{C}_\epsilon$ such that all $x\in u_\epsilon(X)$ and $x\in u_\epsilon(Y)$ satisfy $\max\{|\bb(x)|, |\dd(x)|\}<\tilde{C}_\epsilon$. Let $M:=\max\{d(x, \Delta): x\in X \text{ or } x\in Y\}$. We will show that if $p \in u_\epsilon(Z)$ for some $Z\in S$ then $\max\{ |\bb(p)|, |\dd(p)|\}<\tilde{C}_\epsilon + \sqrt{2}M$.

Consider $p\in Z$ for some $Z\in S$.  This means $p\in \gamma_\phi(t)$ with $\phi\in \Phi$ and $t\in [0,1]$ and hence $p=(1-t)x+t\phi(x)$ for some $x\in X$. We have
\begin{eqnarray*}
\bb(p)&\in [\min\{ \bb(x), \bb(\phi(x))\}, \max \{ \bb(x), \bb(\phi(x))\}]\\
\dd(p)&\in [\min\{ \dd(x), \dd(\phi(x))\}, \max \{ \dd(x), \dd(\phi(x))\}]\\
d(p,\Delta) &\in [\min\{ d(x, \Delta), d(\phi(x), \Delta)\}, \max \{  d(x, \Delta), d(\phi(x), \Delta)\}]\\
\end{eqnarray*}
In order for $d(p,\Delta)\geq \epsilon$ either $d(x,\Delta)\geq \epsilon$ or $d(\phi(x),\Delta)\geq \epsilon$ and hence $\min\{|\bb(x)|, |\bb(\phi(x)|\}<\tilde{C}_\epsilon$ and $\min\{|\dd(x)|, |\dd(\phi(x)|\}<\tilde{C}_\epsilon$.

%
%
%

The condition for $\phi$ to be in $\Phi$ is that $\|x - \Delta\|^2 + \|\phi(x)- \Delta\|^2 \geq \|x-\phi(x)\|^2$ and hence $\|x- \phi(x)\|\leq \sqrt{2}M$. Since $|\bb(x)-\bb(\phi(x))| \leq \|x-\phi(x)\|$ we can conclude that
$$\max\{|\bb(x)|, |\bb(\phi(x)|\} \leq \min\{|\bb(x)|, |\bb(\phi(x)|\}+\sqrt{2}M<\tilde{C}_\epsilon + \sqrt{2} M.$$
Similarly we get $\max\{|\dd(x)|, |\dd(\phi(x)|\} <\tilde{C}_\epsilon + \sqrt{2} M.$

We now will show that $S$ is uniform. Recall that $S$ is uniform if
for all $\epsilon>0$ there exists an $\alpha>0$ such that
$d(l_\alpha(Z),\Delta)<\epsilon$ for all $Z\in S$. For any diagram
$Z\in \D_{L^2}$  denote $M_k(Z)$ as the number of
points  in $Z$ whose distance to the diagonal is in $[2^{-k},
2^{-k+1} )$ for $k\geq1$ and let $M_0(Z)$ be the number points with
distance in $[1, \infty)$. Let $N_k(Z)$ denote the number of points in $Z$ whose distance from the diagonal is at least $2^{-k}$ (in other words the number of off diagonal points in $u_{2^{-k}}(Z)$).

Let $X\cup Y$ be the diagram whose off diagonal points are the union
of the off diagonal points in $X$ and $Y$. Consider the following sum
\begin{align*}
\sum_{j=0}^\infty N_j(X\cup Y) 2^{-2j}&=\sum_{j=0}^\infty \left(\sum_{k=0}^j M_k(X\cup Y)\right) 2^{-2j},\\
&=\sum_{j=0}^\infty M_j(X\cup Y) \left( \sum_{k=j}^\infty 2^{-2k}\right), \\
&= \frac{4}{3}\sum_{j=0}^\infty M_j(X \cup Y) 2^{-2j},\\
&\leq \frac{4}{3}d(X\cup Y, \Delta)^2 <\infty.
\end{align*}

Let $\epsilon>0$. Since $\sum_{j=0}^\infty N_j(X\cup Y) 2^{-2j}$ converges there is some $L$ such that
$$\sum_{j=L}^\infty  N_j(X\cup Y) 2^{-2j}<\epsilon/4.$$ 

Let $\phi \in \Phi$ be a bijection between $X$ and $Y$. Consider the path
$\gamma:[0,1]\to \D_{L^2}$ where $\gamma_\phi(t)$ is the diagram with points
$\{(1-t)x+t\phi(x): x\in X\}$. For the point $(1-t)x+t\phi(x)$ to lie a distance at least $2^{-k}$ from the
diagonal at least one of $x$ or $\phi(x)$ must lie at least $2^{-k}$ from the diagonal. This implies that $N_k(\gamma_\phi(t)) \leq N_k(X\cup Y)$ for all bijections $\phi$ and $t\in [0,1]$. In other words $N_k(Z)\leq N_k(X\cup Y)$ for all $Z\in S$.

Now for any $Z\in S$ we have
\begin{eqnarray*}
d(l_{2^{-L}}(Z), \Delta)^2 \leq \sum_{j=L}^\infty M_j(Z) 2^{-2j+2}\leq  4 \sum_{j=L}^\infty N_j(Z) 2^{-2j}\leq  4 \sum_{j=L}^\infty N_j(X\cup Y) 2^{-2j}<\epsilon.
\end{eqnarray*}
Since the choice of $\alpha=2^{-L}$ was made independently of $Z\in S$ we conclude that $S$ is uniform.

We now know that $\overline{S}$ (the closure of $S$) is compact. Every path $t\mapsto \gamma_\phi(t)$ is a $K_\phi$-Lipschitz map from $[0,1]$ into $\overline{S}$ with $K_\phi^2 = \sum_{x\in X}\|x-\phi(x)\|^2$. 


Set $K= d(X,Y) +1$ and let $A$ be the set of $K$-Lipschitz maps from $[0,1]$ into $\overline{S}$. Since $\overline{S}$ is compact, we know by the Arzela-Ascoli theorem that $A$ is compact. By the definition of the infimum, there exists a sequence of bijections $\{\phi_j\}$ such that $K_{\phi_j}<K$ for all $j$ and $K_{\phi_j}$ is a sequence converging to $K$. The corresponding sequence of paths $\{\gamma_j:=\gamma_{\phi_j}\}$ is a sequence of $K$-Lipschitz maps from $[0,1]$ to $\overline{S}$ and hence lie in the compact set $A$. This means there must be a convergent subsequence of paths $\{\gamma_{n_j}\}$ with some limit $\gamma$ which exists and lies in $A$ as $A$ is compact.

Since $\gamma_{n_j}(0)=X$ and $\gamma_{n_j}(1)=Y$ for all $j$ (as they are all paths from $X$ to $Y$) we know that $\gamma(0)=X$ and $\gamma(1)=Y$. From $d(\gamma_{n_j}(t),\gamma_{n_j}(s)) \leq K_{\phi_{n_j}} |s-t|$ for all $s,t\in [0,1]$ and all $j$  and the limit $K_{\phi_{n_j}} \to K$ as $j\to \infty$ we can infer
$$d(\gamma(t),\gamma(s)) \leq K|s-t|$$
for all $s,t\in [0,1]$. If we follow along the path $\gamma$ where each point $x \in X$ goes to in $Y$ we can construct a bijection $\phi$ from points in $X$ to points in $Y$. This bijection achieves the infimum in \eqref{eq:defdistance}.
\end{proof}

\bibliographystyle{plain}
\bibliography{refs}

\end {document}